\documentclass[centertags,12pt]{amsart}
\usepackage{latexsym}
\usepackage{amsthm}
\usepackage{amssymb}
\usepackage{color}
\usepackage[utf8]{inputenc}
\usepackage{mathrsfs}
\usepackage[english]{babel}
\usepackage{setspace}
\usepackage{ textcomp }
\usepackage{graphicx}
\usepackage{float}
\usepackage{soul}
\usepackage{cancel}
\graphicspath{ {images/} }

\numberwithin{equation}{section}
\makeatletter
\renewcommand{\subsection}{\@startsection
{subsection}{2}{0mm}{\baselineskip}{-0.25cm}
{\normalfont\normalsize\bf}}
\makeatother

\newtheorem{theorem}{Theorem}[section]
\newtheorem{proposition}[theorem]{Proposition}
\newtheorem{lemma}[theorem]{Lemma}

   {\theoremstyle{definition}

}
   \theoremstyle{remark}
\newtheorem{remark}[theorem]{Remark}

\newcommand{\fq}{{\mathbb F_q}}

\newcommand{\cX}{{\mathcal X}}

\newcommand{\cS}{{\mathcal S}}
\newcommand{\aut}{{\rm Aut}}

\begin{document}

\author[Daniele Bartoli, Maria Montanucci and Giovanni Zini]{Daniele Bartoli, Maria Montanucci and Giovanni Zini}
\title{Weierstrass semigroups at every point of the Suzuki curve}

\begin{abstract}
In this article we explicitly determine the structure of the Weierstrass semigroups $H(P)$ for any point
$P$ of the Suzuki curve $\mathcal{S}_q$. As the point $P$ varies, exactly two possibilities arise for $H(P)$: one for the $\mathbb{F}_q$-rational points (already known in the literature), and one for all remaining points. For this last case a minimal set of generators of $H(P)$ is also provided. As an application, we construct dual one-point codes from an $\mathbb{F}_{q^4}\setminus\fq$-point whose parameters are better in some cases than the ones constructed in a similar way from an $\fq$-rational point.
\end{abstract}

\maketitle

\begin{small}

{\bf Keywords:} Suzuki curve, Weierstrass semigroups, algebraic-geometric codes

{\bf 2010 MSC:} Primary: 11G20. Secondary: 11R58, 14H05, 14H55.

\end{small}

\section{Introduction}

Let $\cX$ be a nonsingular, projective algebraic curve of genus $g$ defined over a field $\mathbb{F}$. Let $P$ be a point of $\cX$. The \textit{Weierstrass semigroup} $H(P)$ at $P$ is defined as the set of integers $k$ such that there exists a function on $\cX$ having pole divisor exactly $kP$. It is clear that $H(P)$ is a subset of natural numbers $\mathbb{N}=\{0,1,2,\ldots\}$. By the Weierstrass gap Theorem  \cite[Theorem 1.6.8]{Sti}, the set $G(P):= \mathbb{N} \setminus H(P)$ contains exactly $g$ elements called \textit{gaps}. The structure of $H(P)$ depends on the choice of $P$. However, it is well known that $H(P)$ is the same for all but a finite number of points $P$, namely  the \textit{Weierstrass points} of $\cX$. On one hand,  such points are of intrinsic interest, for example in St\"ohr--Voloch Theory \cite{SV}. On the other hand, in the finite fields setting, the parameters of algebraic-geometric (AG) codes constructed from $\cX$ rely on the inner structure of the semigroup $H(P)$; see e.g. \cite{TV1991}.

 
In this context, \textit{maximal curves}, that is algebraic curves defined over a finite field $\mathbb{F}=\mathbb{F}_q$ having as many rational points as possible according to the Hasse--Weil bound, have been widely investigated. More precisely, an algebraic curve $\cX$ of genus $g$ and defined over $\mathbb{F}_q$ is said to be an $\mathbb{F}_{q}$-maximal curve if it has $q+1+2g\sqrt{q}$ points defined over $\mathbb{F}_q$. Clearly, this can only be the case if either the cardinality $q$ of the finite field is a square or $g=0$.

Among maximal curves, the most studied are the so called \textit{Deligne-Lusztig curves} \cite{DL1976}, that is, the Hermitian curve, the Ree curve in characteristic $3$, and the Suzuki curve
$$\cS_q:\quad Y^q+Y=X^{q_0}(X^q+X)$$
with $q=2q_0^2$, $q_0=2^h$, $h\geq0$.\footnote{OK per mettere l'equazione della sola Suzuki.}

Other examples of maximal curves are  the GK curve \cite{GK2009}, the GGS curve \cite{GGS},  the Klein quartic when $\sqrt{q}\equiv6\pmod7$ \cite{MEAGHER2008}, the BM curve \cite{BM}, together with their quotient curves.

Maximal curves often have large automorphism groups which in many cases can be inherited by the AG code itself: this can bring good performances in encoding \cite{Joyner2005} and decoding \cite{HLS1995}.


In this paper we focus on Weierstrass semigroups at points $P$ of the Suzuki curve $\cS_q$. In particular, we consider the  case $P\in \mathcal{S}_q(\mathbb{F}_{q^4})\setminus \mathcal{S}_q(\mathbb{F}_q)$; that is, $P\in\cS_q$ is $\mathbb{F}_{q^4}$-rational but not $\fq$-rational. Since it is well known that the set of Weierstrass points of $\cS_q$ coincides with $\cS_q(\mathbb{F}_{q})$ (see Lemma \ref{scelgoio}), and $H(P)$ is known for $P\in\cS_q(\fq)$ (see Lemma \ref{notoW}), our assumption $P\in \mathcal{S}_q(\mathbb{F}_{q^4})\setminus \mathcal{S}_q(\mathbb{F}_q)$ is not restrictive. Together with \cite[Lemma 3.1]{M}, our result provides the Weierstrass semigroup at every point of the Suzuki curve. Our main result is the following.

\begin{theorem}\label{Th:Generatori}
Let $P\in\cS_q\setminus\cS_q(\fq)$. The set 
\begin{align}\label{Eq:Generators}
\mathcal{G}&:=&\left\{\nu_{h,k}:=hq-kq_0-\left\lfloor \frac{2h-k-2}{2}\right\rfloor \ : \ h \in \{1,\ldots,q_0\}, k\in \{0,\ldots,2h-2\}\right\}\nonumber\\
&&\cup
 \left\{\mu_{h}:=hq-(2(h-q_0)-1) q_0-(q_0-1) \ : \ h \in \{q_0+1,\ldots,2q_0\}\right\}.
\end{align}
is a minimal set of generators for $H(P)$.
\end{theorem}

It is worth pointing out that, unlike the case $P\in \mathcal{S}_q(\mathbb{F}_q)$, the structure of $H(P)$ with $P\in \mathcal{S}_q\setminus \mathcal{S}_q(\mathbb{F}_q)$ is quite complicated. For instance, the number of generator is far more large, as can be seen comparing  Lemma \ref{notoW} and Theorem \ref{Th:Generatori}.

In the last section we provide examples of dual one-point codes arising from the curve $\cS_q$ and a point $P\in \mathcal{S}_q(\mathbb{F}_{q^4})\setminus \mathcal{S}_q(\mathbb{F}_q)$. We compare such codes with the ones obtained using a point $P^{\prime}\in \mathcal{S}_q(\mathbb{F}_q)$: in some cases improvements on the parameters can be obtained.

\section{Preliminary results}
Through this section and in the rest of the paper we will use the following notation.
Let $q_0=2^s$, where $s \geq 1$, and $q=2q_0^2$. Let us denote the finite field with $q$ elements and its algebraic closure by $\mathbb{F}_q$ and $\mathbb{K}$ respectively. The Suzuki curve $\cS_q$ is given by the affine model
\begin{equation}\label{Eq:Suzuki}
\cS_q:\quad Y^q+Y=X^{q_0}(X^q+X).
\end{equation}
The function field of $\cS_q$ over $\mathbb{K}$ is denoted by $\mathbb{K}(\cS_q)$.
The curve $\cS_q$ is $\mathbb{F}_{q^4}$-maximal of genus $g(\cS_q)=q_0(q-1)$. It has a unique singular point, namely its unique point at infinity $P_\infty$, which is  a $q_0$-fold point and  the centre of just one branch of $\cS_q$. The automorphism group $\aut(\cS_q)$ of $\cS_q$ over  $\mathbb{K}$  is isomorphic to the simple Suzuki group $^2B_2(q)={\rm Sz}(q)$ and it acts on $\cS_q(\mathbb{F}_q)$ as ${\rm Sz}(q)$ on the Suzuki-Tits ovoid in ${\rm PG}(3,q)$. The curve $\cS_q$ is $\mathbb{F}_q$-optimal and its number of $\mathbb{F}_q$-rational points is $|\cS_q(\mathbb{F}_q)|=q^2+1$.
For more details on the main properties of $\cS_q$ we refer the readers to  \cite{FT1,GS}, and  \cite[Section 12.2]{HKT}.

\begin{lemma}{\cite[Example 9.80]{HKT}} \label{fundeq}
Let  $P_0 \in \cS_q(\mathbb{F}_q)$ and $P \in \cS_q$. Then the following linear equivalence holds:
$$qP+2q_0 \Phi(P)+\Phi^2(P) \sim (q+2q_0+1)P_0,$$
where $\Phi$ denotes the $\mathbb{F}_q$-Frobenius endomorphism.
Then for every $P \in \cS_q$  there exists an element $f_P\in \mathbb{K}(\cS_q)$ with principal divisor
$$(f_P)=qP+2q_0 \Phi(P)+\Phi^2(P) - (q+2q_0+1)P_\infty.$$
\end{lemma}

The Weierstrass semigroup at every $\mathbb{F}_q$-rational point of $\cS_q$ is known, as stated in the following lemma.

\begin{lemma}{\cite[Lemma 3.1]{M}} \label{notoW}
The Weierstrass semigroup at every $\mathbb{F}_q$-rational point of $\cS_q$ (and hence in particular at its infinite point $P_\infty$) is
$$H(P_\infty)=\langle q,q+q_0,q+2q_0,q+2q_0+1\rangle.$$
\end{lemma}

The following lemma describes the structure of the set of Weierstrass points of $\cS_q$. As a direct consequence of Lemmas \ref{scelgoio} and \ref{notoW}, in order to compute the Weierstrass semigroup at every point of $\cS_q$, it is not reductive to consider a point $P \in \cS_q(\mathbb{F}_{q^4}) \setminus \cS_q(\mathbb{F}_q)$.

\begin{lemma}{\cite[Remark 5.11]{FT1} \rm{and} \cite[Remark 5.13]{FT1}} \label{scelgoio}
The curve $\cS_q$ is non-classical with respect to the canonical morphism. Furthermore, the set  of Weierstrass points of $\cS_q$ coincides with the set of $\mathbb{F}_q$-rational points $\cS_q(\mathbb{F}_{q})$ of $\cS_q$.
\end{lemma}

\section{The Weierstrass semigroup at every point $P \in \cS_q \setminus \cS_q(\mathbb{F}_q)$}

As recalled in Lemma \ref{notoW}, the structure of $H(P)$ is known if $P \in \cS_q(\mathbb{F}_q)$. Also, due to  Lemma \ref{scelgoio}, it is not reductive to consider $P \in \cS_q(\mathbb{F}_{q^4}) \setminus \cS_q(\mathbb{F}_q)$.
For this reason, for the rest of this section, $P$ will always stand for an $\mathbb{F}_{q^4}$-rational point of $\cS_q$ which is not $\mathbb{F}_q$-rational. Clearly $P$ is an affine point of $\cS_q$, say $P=(a,b)$ with $a,b \in \mathbb{F}_{q^4} \setminus \mathbb{F}_q$.
In order to prove Theorem \ref{Th:Generatori}, we divide our investigation in three steps.

\subsection{The construction}
\ \\ \\
Let $\Phi^i (P)=(a^{q^i},b^{q^i})$ with $i=2,3$ denote the $i$-th image of $P$ with respect to the $\mathbb{F}_q$-Frobenius endomorphism $\Phi:(X,Y) \mapsto (X^q,Y^q)$. From Lemma \ref{fundeq} there exist four functions $f_P$, $f_{\Phi(P)}$, $f_{\Phi^2(P)}$, $f_{\Phi^3(P)}$ in $\mathbb{K}(\cS_q)$ such that
\begin{align} \label{fp}
(f_P)=qP+2q_0 \Phi(P)+\Phi^2(P) - (q+2q_0+1)P_\infty,\\ \label{f1p}
(f_{\Phi(P)})=q\Phi(P)+2q_0 \Phi^2(P)+\Phi^3(P) - (q+2q_0+1)P_\infty,\\ \label{f2p}
(f_{\Phi^2(P)})=q\Phi^2P+2q_0 \Phi^3(P)+P - (q+2q_0+1)P_\infty,\\ \label{f3p}
(f_{\Phi^3(P)})=q\Phi^3(P)+2q_0 P+\Phi(P) - (q+2q_0+1)P_\infty.
\end{align}
Note that since $P\neq P_\infty$, $P$ is a simple point of $\cS_q$ and the tangent line to $\cS_q$ at $P$ is
$$ \mathcal{T}_P: Y-b-a^{q_0}(X-a)=0;$$
denote by $t_P=y-b-a^{q_0}(x-a)\in\mathbb{K}(\cS_q)$ the corresponding rational function.

From  Bezout's Theorem, the intersection of $\cS_q$ and $\mathcal{T}_P$ is a set of $q+q_0$ points (counted with multiplicity) containing both $P$ and $\Phi(P)$. The intersection multiplicity of $\mathcal{T}_P$ and $\cS_q$ at $P$ is $q_0$ while it is equal to $1$ at every other point of intersection. In fact, since $\mathcal{T}_P$ does not contain $P_\infty$, all the other points $Q$ of intersection are simple points of $\cS_q$ and $\mathcal{T}_P$ is not the tangent line of $\cS_q$ at $Q$. Hence,
\begin{equation} \label{tangent}
(t_p)=q_0 P+\Phi(P)+E-(q+q_0)P_\infty,
\end{equation}
where $E$ is the effective divisor of degree $q-1$ whose support consists on the remaining $q-1$ intersection points of  $\cS_q$ and $t_P$.

For non-negative integers $h,i,j,k,$ and $\ell$ we consider the following function  
\begin{equation} \label{funzfam1}
\theta_{h,i,j,k,\ell}:=\frac{f_{\Phi(P)}^i f_{\Phi^2(P)}^j f_{\Phi^3(P)}^k t_P^\ell}{f_P^h} \in \mathbb{K}(\cS_q).
\end{equation}

In what follows, we select suitable functions of type \eqref{funzfam1} to obtain elements in $H(P)$.

\begin{lemma} \label{family1}
Let 
\begin{equation} \label{max}
m_{j,k,\ell}=\bigg\lceil \frac{q_0}{q_0-1}\bigg(j+k+\ell-\frac{k+\ell}{q}\bigg) \bigg\rceil,
\end{equation}
where
\begin{equation}\label{Conditions_lkjh}
\ell\in\{0,1\}, \ k\in\{0,\ldots,q_0-1\}, \ j\in\{0,\ldots,q_0-1\}, \ h\in\{\max\{1,m_{j,k,\ell}\},\ldots,2q_0\}.
\end{equation}
Let 
\begin{equation} \label{nongaps1}
\mathcal{F}_1 := \{\ n_{h,j,k,\ell}:=hq-(\ell+2k)q_0-j \ : \ h,j,k,\ell \ \textrm{ satisfy }\  \eqref{Conditions_lkjh}\}.
\end{equation}
Then $\mathcal{F}_1 \subseteq H(P)$.
\end{lemma} 

\begin{proof}
Let $i$ be a non-negative integer and $\theta_{h,i,j,k,\ell}$ be as in  \eqref{funzfam1}, with $\ell$, $k$, $j$, $h$, $m_{j,k,\ell}$  as  in \eqref{max} and \eqref{Conditions_lkjh}.
Then, from  \eqref{fp}, \eqref{f1p}, \eqref{f2p}, \eqref{f3p}, and \eqref{tangent} one gets
\begin{eqnarray*}
(\theta_{h,i,j,k,\ell})&=&\left( \frac{f_{\Phi(P)}^i f_{\Phi^2(P)}^j f_{\Phi^3(P)}^k t_P^\ell}{f_P^h}\right)\\
&=&i\left(q\Phi(P)+2q_0 \Phi^2(P)+\Phi^3(P) - (q+2q_0+1)P_\infty\right)\\
&&+j\left(q\Phi^2(P)+2q_0 \Phi^3(P)+P - (q+2q_0+1)P_\infty\right)\\
&&+k\left(q\Phi^3(P)+2q_0 P+\Phi(P) - (q+2q_0+1)P_\infty\right)\\
&&+\ell\left(q_0 P+\Phi(P)+E-(q+q_0)P_\infty\right)\\
&&-h\left(qP+2q_0 \Phi(P)+\Phi^2(P) - (q+2q_0+1)P_\infty\right)\\
&=&\ell E+(iq+\ell+k-2hq_0)\Phi(P)+(2iq_0+jq-h)\Phi^2(P)\\
&&+(i+2jq_0+kq)\Phi^3(P)-(hq-(\ell+2k)q_0-j)P\\
&&+[(h-\ell-k-i-j)(q+q_0)+(h-j-k-i)(q_0+1)]P_\infty.
\end{eqnarray*}
Hence, to prove that $n_{h,j,k,\ell} \in H(P)$ for each $(h,j,k,\ell)$ satisfying \eqref{Conditions_lkjh}, it is sufficient to show that $P$ is the unique pole of $\theta_{h,i,j,k,\ell}$ for some non-negative integer $i$. Since the divisor $\ell E \geq 0$ and $i+2jq_0+kq \geq 0$, this is equivalent to show that the following system of inequalities 
\begin{equation}\label{Sistema}
\left\{
\begin{array}{l}
iq+\ell+k-2hq_0 \geq 0\\
2iq_0+jq-h \geq 0\\
(h-\ell-k-i-j)(q+q_0)+(h-j-k-i)(q_0+1) \geq 0 \\
hq-(\ell+2k)q_0-j>0
\end{array}
\right.
\end{equation}
is satisfied for some non-negative integer $i$. The first inequality is equivalent to $i \geq h/q_0-(\ell+k)/q$, while the second is equivalent to $i \geq h/2q_0-j/q_0$. Since $0 \leq k + \ell \leq q_0$,
\begin{equation*}
\frac{h}{q_0}-\frac{\ell+k}{q}\geq \frac{h}{q_0}-\frac{q_0}{q}=\frac{h}{2q_0}+\frac{h-1}{2q_0}\geq \frac{h}{2q_0}+\frac{-2j}{2q_0}=\frac{h}{2q_0}-\frac{j}{q_0}.
\end{equation*}

This means that the first inequality in \eqref{Sistema} implies the second one.
 
The third inequality is equivalent to $i(q+2q_0+1) \leq (h-k-j)(q+2q_0+1)-\ell(q+q_0)$ and, since $i$ must be an integer,
\begin{equation*}
i \leq h-k-j- \bigg\lceil\frac{\ell(q+q_0)}{q+2q_0+1} \bigg\rceil=h-k-j-\ell.
\end{equation*}
Finally, the fourth inequality in \eqref{Sistema} is equivalent to 
\begin{equation*}
h>\frac{j}{q}+\frac{\ell+2k}{2q_0}.
\end{equation*}

First, we show that the above condition on $h$ holds for any $(h,j,k,\ell)$ satisfying \eqref{Conditions_lkjh}.
 
Note that $\max\{1,m_{j,k,\ell}\}=1$ if and only if $j=k=\ell=0$: if $j=k=\ell=0$ then 
$$m_{j,k,\ell}=\left\lceil\frac{q_0}{q_0-1}\bigg(j+k+\ell-\frac{k+\ell}{q}\bigg)\right\rceil =0;$$
if $j+k+\ell\geq1$ then 
$$m_{j,k,\ell}=\left\lceil\frac{q_0}{q_0-1}\bigg(j+k+\ell-\frac{k+\ell}{q}\bigg)\right\rceil\geq \left\lceil\frac{q_0}{q_0-1}\right\rceil\geq2.$$



Consider now the case $j+k+\ell \geq 1$. We have that 
\begin{align*}
q q_0(j+k+\ell)-(k+\ell)q_0-(q_0-1)(j+q_0 \ell+2q_0 k)\\
=k(qq_0-q_0-q+2q_0)+j(qq_0-q_0+1)+\ell(qq_0-q_0^2) \\
\geq k+j+\ell \geq 1,
\end{align*}
and hence  
$$\frac{q_0}{q_0-1}\bigg(j+k+\ell-\frac{k+\ell}{q}\bigg) \geq \frac{j}{q}+\frac{\ell+2k}{2q_0},$$
which yields 
$$m_{j,k,\ell} \geq \frac{j}{q}+\frac{\ell+2k}{2q_0}$$
and shows that the fourth inequality in \eqref{Sistema} is satisfied for any $(h,j,k,\ell)$ satisfying \eqref{Conditions_lkjh}.

Summing up, we only have to show that there exists at least one non-negative integer $i$ with 
$$\frac{h}{q_0}-\frac{\ell+k}{q} \leq i \leq h-k-j-\ell.$$
Since $k+j+\ell \leq \max\{1,m_{j,k,\ell}\}\leq h$, the integer $h-k-j-\ell$ is non-negative and therefore it is enough to show that $\frac{h}{q_0}-\frac{\ell+k}{q} \leq h-k-j-\ell$, which is equivalent to  
$$h \geq \frac{q_0}{q_0-1}\bigg(k+\ell+j-\frac{k+\ell}{q}\bigg).$$
As $h\geq m_{j,k,\ell}$, the claim follows.
\end{proof}

We now construct a second family of nongaps at $P$ which we will prove to be disjoint from $\mathcal{F}_1$.
\begin{lemma} \label{family2}
Let 
\begin{equation} \label{nongaps2}
\mathcal{F}_2 := \{\ n_{\tilde h}=\tilde{h}q-(2\tilde{h}-2q_0-1)q_0-(q_0-1) \ : \ \tilde h\in\{q_0+1,\ldots,2q_0\}\ \}.
\end{equation}
Then $\mathcal{F}_2 \subseteq H(P)$.
\end{lemma}

\begin{proof}
Let $\tilde i\geq0$, $\tilde{h}\in\{q_0+1,\ldots,2q_0\}$, and
\begin{equation} \label{funzfam2}
\psi_{\tilde{h},\tilde{i}}:=\frac{f_{\Phi(P)}^{\tilde i} f_{\Phi^3(P)}^{\tilde h-q_0}}{f_P^{\tilde h} f_{\Phi^2(P)}} \in \mathbb{K}(\cS_q).
\end{equation}
Then from \eqref{fp}, \eqref{f1p}, \eqref{f2p}, and \eqref{f3p} follows
\begin{eqnarray*}
(\psi_{\tilde{h},\tilde{i}})&=&\left(\frac{f_{\Phi(P)}^{\tilde i} f_{\Phi^3(P)}^{\tilde h-q_0}}{f_P^{\tilde h} f_{\Phi^2(P)}}\right)\\
&=&\tilde{i}\left(q\Phi(P)+2q_0 \Phi^2(P)+\Phi^3(P) - (q+2q_0+1)P_\infty\right)\\
&&+(\tilde h -q_0)\left(q\Phi^3(P)+2q_0 P+\Phi(P) - (q+2q_0+1)P_\infty\right)\\
&&-\tilde{h}\left(qP+2q_0 \Phi(P)+\Phi^2(P) - (q+2q_0+1)P_\infty\right)\\
&&-\left(q\Phi^2(P)+2q_0 \Phi^3(P)+P - (q+2q_0+1)P_\infty\right)\\
&=&(\tilde{i}q+\tilde{h}-q_0-2\tilde{h}q_0)\Phi(P)+(2\tilde{i}q_0-\tilde{h}-q)\Phi^2(P)\\
&&+(\tilde{i}+\tilde{h}q-q q_0-2q_0)\Phi^3(P)+(-\tilde{i}+q_0+1)(q+2q_0+1)P_\infty\\
&&-(\tilde{h}q-(2\tilde{h}-2q_0-1)q_0-(q_0-1))P.
\end{eqnarray*}

Hence, as in the proof of Lemma \ref{family1}, the claim follows proving that $\psi_{\tilde{h},\tilde{i}}$ has a unique pole at $P$. This is equivalent to prove that the following system of inequalities
\begin{equation}\label{Sistema2}
\left\{
\begin{array}{l}
\tilde{i}q+\tilde{h}-q_0-2\tilde{h}q_0 \geq 0 \\
2\tilde{i}q_0-\tilde{h}-q \geq 0 \\
\tilde{i}+\tilde{h}q-q q_0-2q_0 \geq 0 \\
-\tilde{i}+q_0+1 \geq 0 \\
\tilde{h}q-(2\tilde{h}-2q_0-1)q_0-(q_0-1) >0
\end{array}
\right.
\end{equation}
is satisfied for some non-negative integer $\tilde{i}$.

The first and the second inequalities in \eqref{Sistema2} are equivalent to $\tilde{i} \geq \tilde{h}/q_0+1/2q_0-\tilde{h}/q$ and $\tilde{i} \geq \tilde{h}/2q_0+1/q_0$, respectively. 
Since by hypothesis $\tilde h\geq q_0+1>1$, 
$$\frac{\tilde h}{2q_0}+\frac{1}{q_0}=\frac{\tilde h}{q_0}+\frac{2-\tilde h}{2q_0}\leq \frac{\tilde h}{q_0}+\frac{1-\tilde h/ q_0}{2q_0},$$
and the first inequality implies the second one.

The third inequality is equivalent to $\tilde{i} \geq 2q_0+qq_0-\tilde{h}q$, which is always satisfied since $\tilde h\geq q_0+1$.
The fourth inequality  is equivalent to $\tilde{i} \leq q_0+1$. Finally, the last inequality in \eqref{Sistema2} is equivalent to $\tilde{h}(q-2q_0)>-(2q_0+1)q_0+(q_0-1)=-q-1$ which is satisfied as $\tilde{h}$ is strictly positive. 

The integer  $\tilde{i}:=q_0+1$ is positive and clearly satisfies  $\tilde i \geq 3> \tilde{h}/q_0+1/2q_0.$
The claim follows.
\end{proof}

\subsection{$\langle \mathcal{F}_1 \cup \mathcal{F}_2 \rangle$ coincides with $H(P)$}
\ \\ \\

\begin{remark}\label{maria}
Let $S$ be a numerical semigroup and $s$ be the multiplicity of $S$, i.e. its least non-zero element. If an element $r\in S$ is such that $[r,\ldots,r+s-1]\subset S$, then the conductor of $S$ is at most $r$, that is $m\in S$ for any $m\geq r$; in fact $m=\left\lfloor (m-r)/s\right\rfloor\cdot s + u$ for some $u\in[r,\ldots,r+s-1]$.
\end{remark}
By  Lemma \ref{family1} and Lemma \ref{family2} we already know that  $\langle \mathcal{F}_1 \cup \mathcal{F}_2 \rangle \subseteq H(P)$. In what follows we prove that $\mathcal{F}_1 \cup \mathcal{F}_2$ contains at least $g(\cS_q)$ elements which are less than or equal to $2g(\cS_q)-1$, and that $[2g(\cS_q)-q+2,\ldots,2g(\cS_q)+1]\subset \mathcal{F}_1 \cup \mathcal{F}_2$. Therefore $H(P)=\langle \mathcal{F}_1 \cup \mathcal{F}_2 \rangle$ by Remark \ref{maria}.

\begin{lemma} \label{contofamily1}
Let $\mathcal{F}_1$ be as in \eqref{nongaps1}. Then the elements of $\mathcal{F}_1$ are pairwise distinct and

\begin{equation*}
|\mathcal{F}_1 \cap [1,\ldots,2g(\cS_q)-1]|=2q_0^3-2q_0-1.
\end{equation*}
\end{lemma}

\begin{proof}
Suppose $n_{h,j,k,\ell}=n_{h_1,j_1,k_1,\ell_1}\in \mathcal{F}_1$, for some $(h,j,k,\ell),(h_1,j_1,k_1,\ell_1)$ satisfying \eqref{Conditions_lkjh}. Then $hq-(\ell+2k)q_0-j=h_1 q-(\ell_1+2k_1)q_0-j_1$ and $j \equiv j_1 \pmod {q_0}$. Since $0 \leq j,j_1 \leq q_0-1$ we have that $j=j_1$ and hence $2hq_0-(\ell+2k)=2q_0 h_1-(\ell_1+2k_1)$. Thus $\ell \equiv \ell_1 \pmod 2$ and therefore $\ell=\ell_1$. We are left with $hq_0-k=q_0 h_1-k_1$. Considering the congruence modulo $q_0$ we get that $k=k_1$ and hence $h=h_1$. This shows that there are no repetitions in $\mathcal{F}_1$.

Let $n_{h,j,k,\ell} \in \mathcal{F}_1$ such that $n_{h,j,k,\ell} \geq 2g(\cS_q)=2q_0 q-2q_0$. This yields $h=2q_0$ and $2k+\ell \leq 2$. More precisely, one of the following three possibilities occurs
\begin{itemize}
\item $\ell=0$, $k=1$ and $j=0$, or
\item $\ell=1$, $k=0$ and $j=0,\ldots,q_0-1$, or
\item $\ell=0$, $k=0$ and $j=0,\ldots,q_0-1$.
\end{itemize}
Since $0\notin \mathcal{F}_1$, 
\begin{equation}\label{ContoF1}
|\mathcal{F}_1 \cap [1,\ldots,2g(\cS_q)-1]|=|\mathcal{F}_1|-1-q_0-q_0=|\mathcal{F}_1|-2q_0-1.
\end{equation}

To compute $|\mathcal{F}_1|$ it is convenient to divide the elements of $\mathcal{F}_1$ according to the corresponding value of $M_{j,k,\ell}:=\max\{1,m_{j,k,\ell}\}$. To this end, three cases are considered.

\begin{enumerate}

\item[\rm{(A)}] Assume that $j=k=\ell=0$. Then, clearly, $m_{j,k,\ell}=0$. In this case $M_{j,k,\ell}=1$.

\item[\rm{(B)}] Let $j+k+\ell\geq1$ and $j \leq q_0-1-(k+\ell)$. Then
\begin{eqnarray*}
m_{j,k,\ell}&=&\bigg\lceil \frac{q_0}{q_0-1}\bigg(j+k+\ell-\frac{k+\ell}{q}\bigg) \bigg\rceil=j+k+\ell+\bigg\lceil \frac{j+k+\ell}{q_0-1}-\frac{k+\ell}{2q_0(q_0-1)} \bigg\rceil\\
&\leq& j+k+\ell+\bigg\lceil \frac{(q_0-1-k-\ell)+k+\ell}{q_0-1}-\frac{k+\ell}{2q_0(q_0-1)} \bigg\rceil=j+k+\ell+1.\\
\end{eqnarray*}
Since $j+k+\ell\geq 1$ and then $m_{j,k,\ell} \geq j+k+\ell+1$, we get  $M_{j,k,\ell}=m_{j,k,\ell} =j+k+\ell+1$.

\item[\rm{(C)}] Let $j \geq q_0-(k+\ell)$. Then $j+k+\ell\geq1$. Also,
\begin{eqnarray*}
m_{j,k,\ell}&= & j+k+\ell+\bigg\lceil \frac{j+k+\ell}{q_0-1}-\frac{k+\ell}{2q_0(q_0-1)} \bigg\rceil\\
&\geq&j+k+\ell+\bigg\lceil \frac{q_0}{q_0-1}-\frac{1}{2(q_0-1)} \bigg\rceil >j+k+\ell+1;\\
m_{j,k,\ell}&=&j+k+\ell+\bigg\lceil \frac{j+k+\ell}{q_0-1}-\frac{k+\ell}{2q_0(q_0-1)} \bigg\rceil\\
 &\leq& j+k+\ell+\bigg\lceil \frac{2q_0-1}{q_0-1} \bigg\rceil=j+k+\ell+3.\\
\end{eqnarray*}
Hence either $m_{j,k,\ell}=j+k+\ell+2$ or $m_{j,k,\ell}=j+k+\ell+3$. Actually, the case $m_{j,k,\ell}=j+k+\ell+3$ cannot occur. Indeed, $m_{j,k,\ell}=j+k+\ell+3$ if and only if
$$\bigg\lceil \frac{j+k+\ell}{q_0-1}-\frac{k+\ell}{2q_0(q_0-1)} \bigg\rceil=3\quad  \iff\quad  j+k+\ell>\frac{k+\ell}{2q_0}+2q_0-2 >2q_0-2.$$
This yields $j+k+\ell=2q_0-1$ and therefore
  $$h\geq M_{j,k,\ell}=m_{j,k,\ell}=j+k+\ell+3=2q_0+2>2q_0,$$ a contradiction to the hypothesis $h \leq 2q_0$. 
So  $M_{j,k,\ell}=m_{j,k,\ell}=j+k+\ell+2$.
\end{enumerate}
We proceed computing the number of elements in $\mathcal{F}_1$ of type (A), (B) and (C).
\begin{itemize}
\item Since $h=1,\ldots, 2q_0$, there are exactly $2q_0$ elements of type (A).
\item Deleting the $2q_0$ elements of type (A), for which clearly $j \leq q_0-1-(k+\ell)$, and noting that $k \leq q_0-1-\ell$ as $j \geq 0$, we get that the number of elements of type (B) is
\begin{align*}
-2q_0+\sum_{\ell=0}^{1} \sum_{k=0}^{q_0-1-\ell} \sum_{j=0}^{q_0-1-k-\ell}(2q_0-j-k-\ell)=\\
-2q_0+\sum_{\ell=0}^{1}\sum_{k=0}^{q_0-1-\ell}\frac{(q_0-k-\ell)(3q_0-k-\ell+1)}{2}.\\
\end{align*}
From the previous equality, defining $x=q_0-k-\ell$, we get that the number of elements of type (B) equals
\begin{align*}
-2q_0+\sum_{\ell=0}^{1}\sum_{x=1}^{q_0-\ell}\left(\frac{x^2}{2}+ x\left(q_0+\frac{1}{2}\right)\right)=\\
-2q_0+\sum_{\ell=0}^{1}\left(\bigg(q_0+\frac{1}{2}\bigg)\frac{(q_0-\ell)(q_0-\ell+1)}{2}+\frac{1}{2} \sum_{x=1}^{q_0-\ell}x^2\right)=\\
-2q_0+\sum_{\ell=0}^{1}\bigg[ \bigg(q_0+\frac{1}{2}\bigg)\frac{(q_0-\ell)(q_0-\ell+1)}{2}+\frac{(q_0-\ell)(q_0-\ell+1)(2q_0-2\ell+1)}{12}\bigg]=\\
\frac{4}{3}q_0^3+\frac{1}{2}q_0^2+\frac{1}{6}q_0-2q_0.\\
\end{align*}

\item We divide the elements of type (C) in two classes, namely the ones having $\ell=0$ and the ones for which $\ell=1$.

If $\ell=0$ then $1\leq k\leq q_0-1$ as $q_0-k \leq j \leq q_0-1$, while $j+k+\ell+2\leq h\leq 2q_0$. Hence the number $c_0$ of elements of type (C) with $\ell=0$ is
\begin{align*}
c_0=\sum_{k=1}^{q_0-1} \sum_{j=q_0-k}^{q_0-1}(2q_0-j-k-1)=\sum_{k=1}^{q_0-1}\frac{(2q_0-k-1)k}{2}=\\
\sum_{k=1}^{q_0-1}\left(k\left(q_0-\frac{1}{2}\right)-\frac{k^2}{2}\right)=\left(q_0-\frac{1}{2}\right)\frac{(q_0-1)q_0}{2}-\frac{q_0(q_0-1)(2q_0-1)}{12}=\\
\frac{q_0(q_0-1)}{2}\left(q_0-\frac{1}{2}-\frac{2q_0-1}{6}\right).\\
\end{align*}

If $\ell=1$ then $0\leq k\leq q_0-1$, $q_0-1-k\leq j\leq q_0-1$ and $j+k+3\leq h\leq 2q_0$. Thus the number $c_1$ of elements of type (C) with $\ell=1$ is

\begin{align*}
c_1=\sum_{k=0}^{q_0-1} \sum_{j=q_0-k-1}^{q_0-1}(2q_0-j-k-2)=\sum_{k=0}^{q_0-1}\frac{(2q_0-k-2)(k+1)}{2}=\\
\sum_{k=0}^{q_0-1}\left(k\left(q_0-\frac{3}{2}\right)-\frac{k^2}{2}+q_0-1\right)=\left(q_0-\frac{3}{2}\right)\frac{(q_0-1)q_0}{2}-\frac{q_0(q_0-1)(2q_0-1)}{12}+q_0^2-q_0=\\
\frac{q_0(q_0-1)}{2}\left(q_0+\frac{1}{2}-\frac{2q_0-1}{6}\right).\\
\end{align*}

By direct computation, we get 
$$c_0+c_1=\frac{q_0(q_0-1)}{2} \left(2q_0-\frac{2q_0-1}{3}\right)=\frac{2}{3}q_0^3-\frac{q_0^2}{2}-\frac{q_0}{6}.$$

\end{itemize}
Summing up the contributions we get that
$$|\mathcal{F}_1|=2q_0+\frac{4}{3}q_0^3+\frac{1}{2}q_0^2+\frac{1}{6}q_0-2q_0+\frac{2}{3}q_0^3-\frac{1}{2}q_0^2-\frac{1}{6}q_0=
2q_0^3,$$
and hence, by \eqref{ContoF1}, 
$$|\mathcal{F}_1 \cap [1,\ldots,2g(\cS_q)-1]|=|\mathcal{F}_1|-2q_0-1=2q_0^3-2q_0-1.$$
\end{proof}

The following lemma shows that $\mathcal{F}_1$ and $\mathcal{F}_2$ are disjoint.

\begin{lemma} \label{disj}
Let $\mathcal{F}_1$ and $\mathcal{F}_2$ be the sets of non-gaps defined in Lemma {\rm \ref{family1}} and Lemma {\rm \ref{family2}} respectively. Then $\mathcal{F}_1 \cap \mathcal{F}_2=\emptyset$.
\end{lemma}

\begin{proof}
Assume by contradiction that $n_{h,j,k,\ell}=n_{\tilde{h}}$ for some $n_{h,j,k,\ell} \in \mathcal{F}_1$ and $n_{\tilde h} \in \mathcal{F}_2$. Then
$$hq-(\ell+2k)q_0-j=\tilde{h}q-(2\tilde{h}-2q_0-1)q_0-(q_0-1).$$
Since $j$ must be congruent to $-1$ modulo $q_0$ we get that $j=q_0-1$ and 
$$2hq_0-(\ell+2k)=2\tilde{h}q_0-(2\tilde{h}-2q_0-1).$$
Considering the congruence modulo $2$ we get that $\ell=1$ and
$$hq_0-k=\tilde{h}q_0-\tilde{h}+q_0+1.$$
Taking the congruence modulo $q_0$ we get that $k=\tilde{h}-(q_0+1)$ and $h=\tilde{h}$; hence, $(h,j,k,\ell)=(\tilde{h},q_0-1,\tilde{h}-(q_0+1),1)$. Conditions \eqref{Conditions_lkjh} imply 
\begin{eqnarray*}
h &\geq& m_{j,k,\ell} \geq \frac{q_0}{q_0-1}\bigg(j+k+\ell-\frac{k+\ell}{q}\bigg)\\
&=&\frac{q_0}{q_0-1}\bigg(q_0-1+\tilde h-(q_0+1)+1-\frac{\tilde h-(q_0+1)+1}{q} \bigg)\\
&=&\frac{q_0}{q_0-1} \bigg( \tilde h-1-\frac{\tilde h-q_0}{q}\bigg)= \tilde h \frac{2q_0^2-1}{2q_0(q_0-1)} -\frac{q_0-1/2}{q_0-1}>\tilde h,\\
\end{eqnarray*}
a contradiction to $(h,j,k,\ell)=(\tilde{h},q_0-1,\tilde{h}-(q_0+1),1)$.
\end{proof}

\begin{lemma} \label{contofamily2} 
Let $\mathcal{F}_2$ be the set of non-gaps constructed in Lemma {\rm \ref{family2}}. Then the elements of $\mathcal{F}_2$ are pairwise distinct and

\begin{equation*}
|\mathcal{F}_2 \cap [1,\ldots,2g(\cS_q)-1]|=q_0.
\end{equation*}
\end{lemma}

\begin{proof}
It is easily seen that  $n_{\tilde h}=n_{\tilde{h}_1}$ implies $\tilde{h}=\tilde{h}_1$ and therefore the elements of $\mathcal{F}_2$ are pairwise distinct. Also, $\mathcal{F}_2$ contains exactly $q_0$ elements since $\tilde{h}=q_0+1,\ldots,2q_0$. Finally, for each $\tilde{h}=q_0+1,\ldots,2q_0$, 
\begin{eqnarray*}
n_{\tilde{h}}&=&\tilde{h}q-(2\tilde{h}-2q_0-1)q_0-(q_0-1) \leq 2q_0 q-(4q_0-2q_0-1)q_0-(q_0-1)\\
&=&2q_0 q -(2q_0-1)q_0-q_0+1=2q_0 q -(q-1)<2g(\cS_q)-1.
\end{eqnarray*}
The claim follows.
\end{proof}

\begin{lemma}\label{belpezzo}
The interval $[2g(\cS_q)-q+2,\ldots,2g(\cS_q)+1]$ is contained in $\langle\mathcal{F}_1 \cup \mathcal{F}_2\rangle$.
\end{lemma}
\begin{proof}
Let $I=[2g(\cS_q)-q+2,\ldots,2g(\cS_q)+1]=[(2q_0-1)q-q_0-(q_0-2),\ldots,2q_0 q-2q_0+1]$ and $t\in I$.
\begin{itemize}
\item
If $t=2q_0 q-q+1$, then $t=n_{2q_0}\in\mathcal{F}_2$.
\item
If $t\in I$ and $t\leq(2q_0-1)q$, then $t=n_{2q_0-1,j,0,\ell}$. Since $m_{j,0,\ell}\leq j+0+\ell+2<q_0+1\leq2q_0-1=h$, \eqref{Conditions_lkjh} is satisfied and $t\in\mathcal{F}_1$.
\item
If $t\in I\setminus\{2q_0 q-q+1\}$ and $t>(2q_0-1)q$, then $t=n_{2q_0,l,k,\ell}$ where $m_{j,k,\ell}\leq j+k+\ell+2\leq 2q_0=h$ by direct checking. Hence \eqref{Conditions_lkjh} is satisfied and $t\in\mathcal{F}_1$.
\end{itemize}
\end{proof}
\begin{remark}\label{nonsimmetrico}
From Lemma \ref{belpezzo} it follows in particular that $2g(\cS_q)-1$ is a non-gap at $P$, that is the semigroup $H(P)$ is non-symmetric for any $P\in\cS_q\setminus\cS_q(\fq)$; unlikely the Weierstrass semigroup at any $\fq$-rational point of $\cS_q$, which is symmetric (\cite[Lemma 5.7]{FT1}). 
\end{remark}

\begin{theorem} \label{finesect}
Let $\mathcal{F}_1$ and $\mathcal{F}_2$ as in Lemmas {\rm \ref{family1}} and {\rm \ref{family2}}. If $P \in \cS_q$ is not $\mathbb{F}_q$-rational then
\begin{equation*}
H(P)=\langle \mathcal{F}_1 \cup \mathcal{F}_2 \rangle.
\end{equation*}
\end{theorem}
\begin{proof}
The claim follows immediately from Lemma \ref{belpezzo} and
$$|(\mathcal{F}_1 \cup \mathcal{F}_2\cup \{0\}) \cap [0,\ldots,2g(\cS_q)-1] |=2q_0^3-q_0=q_0(q-1)=g(\cS_q).$$
\end{proof}

\subsection{A minimal set of generators for $H(P)$}
\ \\


In this section we provide a set of minimal generators for the Weierstrass semigroup $H(P)$, $P\in \mathcal{S}_q(\mathbb{F}_{q^4})\setminus \mathcal{S}_q(\mathbb{F}_q)$. The proof of Theorem \ref{Th:Generatori} is divided into two steps.

Let us define 
\begin{equation}\label{Eq:GeneratorFirstType}
\mathcal{G}_1 := \left\{\nu_{h,k}=hq-k q_0-\left\lfloor \frac{2h-k-2}{2}\right\rfloor \ : \ h \in \{1,\ldots,q_0\}, k\in \{0,\ldots,2h-2\}\right\},
\end{equation}
so that $\mathcal{G}=\mathcal{G}_1 \cup \mathcal{F}_2$.

\textbf{Step 1: $\mathcal{G}$ generates $H(P)$.}
\ \\ \\
Consider an element $n_{h,j,k,\ell}\in \mathcal{F}_1$. Let $\delta(n_{h,j,k,\ell})=\left\lfloor \frac{2h-(\ell+2k)-2}{2}\right\rfloor-j=h-1-k-j-\left\lceil \frac{\ell}{2}\right\rceil$. We will relate $\delta(n_{h,j,k,\ell})$ with the number of elements of $\mathcal{G}_1$ we need to generate $n_{h,j,k,\ell}$.

\begin{proposition}\label{delta0}
If $\delta(n_{h,j,k,\ell})=0$ then $n_{h,j,k,\ell}$ is a generator. 
\end{proposition}
\begin{proof}
If  $\delta(n_{h,j,k,\ell})=0$ then $h \leq q_0$. In fact, $$\frac{q_0}{q_0-1}\bigg(j+k+\ell-\frac{k+\ell}{q} \bigg) \leq h.$$ Using $j=\left\lfloor \frac{2h-(\ell+2k)-2}{2}\right\rfloor$ we get $$\frac{q_0}{q_0-1}\bigg(\left\lfloor \frac{2h-(\ell+2k)-2}{2}\right\rfloor+k+\ell-\frac{k+\ell}{q} \bigg) \leq h,$$ and hence $$\frac{q_0}{q_0-1}\bigg(h-1-\left\lceil \frac{\ell}{2} \right\rceil+\ell -\frac{k+\ell}{q} \bigg) \leq h.$$ Since $\ell-\left\lceil \frac{\ell}{2} \right\rceil=0$, we get $$h \leq q_0+\left\lfloor \frac{k+\ell}{2q_0} \right\rfloor=q_0.$$ Now to show that if $\delta(n_{h,j,k,\ell})=0$ then $n_{h,j,k,\ell}$ is a generator it is sufficient to note that $n_{h,j,k,\ell}=\nu_{h,2k+\ell}$, $h \leq q_0$ and $2h-(\ell+2k)-2 \geq 0$ implying $2k+\ell \leq 2h-2$.
\end{proof}

So we can assume that $\delta(n_{h,j,k,\ell}) \ne 0$. We distinguish two cases.

\begin{proposition}\label{n_{h,j,k,0}}
Every element $n_{h,j,k,0}\in \mathcal{F}_1$ belongs to $\langle \mathcal{G}_1\rangle$.
\end{proposition}
\begin{proof}
We prove that we need $\delta(n_{h,j,k,0})+1$ elements from $\mathcal{G}_1$ to generate $n_{h,j,k,0}$. First, we show  that $0 \leq \delta(n_{h,j,k,0})\leq h-1$. Since $h-1=\delta(n_{h,j,k,0})+k+j$ the upper bound is clear. To see the lower bound, recall that
\begin{equation}\label{LowerBound_h}
h\geq \left\lceil\frac{q_0}{q_0-1}\left(j+k+\ell-\frac{k+\ell}{q}\right)\right\rceil,
\end{equation}
so 
$$\delta(n_{h,j,k,0})\geq \left\lceil\frac{q_0}{q_0-1}\left(j+k-\frac{k}{q}\right)\right\rceil-1-k-j.$$
Now, apart from the case $j=k=0$, we have that 
$\frac{q_0}{q_0-1}\left(j+k-\frac{k}{q}\right)>k+j$, since $j+k>\frac{k}{2q_0}$.  
If $j=k=0$ then $h \geq 1=\max\left\{1,\left\lceil\frac{q_0}{q_0-1}\left(j+k+\ell-\frac{k+\ell}{q}\right)\right\rceil\right\}$ and hence $\delta(n_{h,j,k,0}) \geq 0$.

Consider a sequence of $(h_i,k_i)_{i=1,\ldots,\delta(n_{h,j,k,0})+1}$ such that 
$$\sum_{i=1}^{\delta(n_{h,j,k,0})+1} h_i=h, \quad q_0 \geq h_i\geq 1, \quad \sum_{i=1}^{\delta(n_{h,j,k,0})+1} k_i =2k, \quad k_i \textrm{ even}.$$ 
By \eqref{LowerBound_h}, we can consider $h_i\ge k_i/2+1$, since $h\geq k+1$ and then each $\nu_{h_i,k_i}$ is an element of $\mathcal{G}_1$. 
Now,
\begin{eqnarray*}
\sum_{i=1}^{\delta(n_{h,j,k,0})+1} \left(h_iq-k_i q_0-\left\lfloor \frac{2h_i-k_i-2}{2}\right\rfloor\right)&=&\sum_{i=1}^{\delta(n_{h,j,k,0})+1} \left( h_i q-k_iq_0-h_i+1+\frac{k_i}{2}\right)\\
&=&hq -2k q_0 -h+\delta(n_{h,j,k,0})+1+k\\
&=&hq -2k q_0 -j=n_{h,j,k,0}
\end{eqnarray*}
and the claim follows.
\end{proof}

\begin{proposition}\label{n_{h,j,k,1}}
Every element  $n_{h,j,k,1}\in \mathcal{F}_1$ belongs to $\langle \mathcal{G}_1\rangle$.
\end{proposition}
\begin{proof}
We distinguish three subcases. 
\begin{enumerate}
\item Suppose $h\geq m_{j,k,1}+1$.  We can also suppose $h\geq 3$, since the smallest integer with $\ell=1$ which is not in $\mathcal{G}_1$  is $3q-q_0$. 
Consider the generator $\nu_{2,1}=2q-q_0$. Now 
$$n_{h,j,k,1}-\nu_{2,1}=n_{h-2,j,k,0}.$$
We only have to prove that such element is in the semigroup $H(P)$ so that, by Proposition \ref{n_{h,j,k,0}}, it belongs to $\langle \mathcal{G}_1\rangle$. To this end, observe that
$$h-2\geq m_{j,k,1}-1\geq m_{j,k,0}$$
because
\begin{eqnarray*}
m_{j,k,1}-1&= &\left\lceil \frac{q_0}{q_0-1}\left(j+k+1-\frac{k+1}{q}\right) \right\rceil-1\\
&=&\left\lceil \frac{q_0}{q_0-1}\left(j+k-\frac{k}{q}\right) +\frac{q_0}{q_0-1}-\frac{1}{2q_0(q_0-1)}\right\rceil-1\\
&=&\left\lceil \frac{q_0}{q_0-1}\left(j+k-\frac{k}{q}\right) +\frac{2q_0-1}{q-2q_0}\right\rceil \geq m_{j,k,0}.
\end{eqnarray*}
Therefore $n_{h,j,k,1}\in \langle \mathcal{G}_1\rangle$.

\item Suppose $h= m_{j,k,1}$ and $h\leq q_0$; we show that $n_{h,j,k,1}$ is a generator.
From $h\leq q_0$ follows $j+k+1\leq q_0-1$.  
Also,
\begin{eqnarray*}
\delta&=& m_{j,k,1}-k-j-2\leq \frac{q_0}{q_0-1}\left(j+k+1-\frac{k+1}{q}\right)-(k+j+1)\\
&=& \frac{1}{q_0-1}(j+k+1)-\frac{k+1}{2q_0(q_0-1)}\leq 1-\frac{k+1}{2q_0(q_0-1)}<1,\\
\end{eqnarray*}
and then $\delta=0$, so that $n_{h,j,k,1}$ is a generator by Proposition \ref{delta0}.

\item Suppose $h= m_{j,k,1}=q_0+i$, $1\leq i\leq q_0$.

First, note that $j+k=q_0+i-3$. In fact,

\begin{eqnarray*}
q_0+i-1\le& \frac{q_0}{q_0-1}\left(j+k+1-\frac{k+1}{q}\right)&\leq q_0+i,\\
q_0+i-1+\frac{k+1}{2q_0(q_0-1)}\le& \frac{q_0}{q_0-1} (j+k+1)&\leq q_0+i+\frac{k+1}{2q_0(q_0-1)},\\
\frac{q_0-1}{q_0} (q_0+i-1)+\frac{k+1}{q}\le& j+k+1&\leq \frac{q_0-1}{q_0} (q_0+i)+\frac{k+1}{q},\\
q_0-1+\frac{q_0-1}{q_0} (i-1)+\frac{k+1}{q}\le& j+k+1&\leq q_0-1+\frac{q_0-1}{q_0} i+\frac{k+1}{q}.\\
\end{eqnarray*}

Now, $\frac{q_0-1}{q_0} (i-1)+\frac{k+1}{q}>i-2$ since $-\frac{i-1}{q_0}+\frac{k+1}{q}>-1$, and $\frac{q_0-1}{q_0}i+\frac{k+1}{q}<i$, since $-\frac{i}{q_0}+\frac{k+1}{q}<0$. So $q_0-1+i-2<j+k+1<q_0-1+i$ and then $j+k+1=q_0-2+i$, that is $j+k=q_0-3+i$. 

This is enough to show that actually $i=1$ cannot occur. In fact, 
\begin{eqnarray*}
h&=&q_0+1>q_0=\left\lceil\frac{q_0}{q_0-1}\left(q_0-1-\frac{k+1}{q}\right)\right\rceil\\
&=&   \left\lceil\frac{q_0}{q_0-1}\left(j+k+1-\frac{k+1}{q}\right)\right\rceil= m_{j,k,1},
\end{eqnarray*}
a contradiction to $h= m_{j,k,1}$.

Since $i\geq 2$,  $j=q_0-3+i-k\geq i-2$. 

We prove now that  $n_{h,j,k,1}-\nu_{i,1}$ belongs to the semigroup. In fact, $n_{h,j,k,1}-\nu_{i,1}=n_{q_0,j-i+2,k,0}$ so that, in order to prove that  $n_{q_0,j-i+2,k,0}\in \mathcal{F}_1$, we only have to check that $q_0\geq m_{j-i+2,k,0}$. We have that 
\begin{eqnarray*}
m_{j-i+2,k,0}&=&\left\lceil\frac{q_0}{q_0-1}\left(j-i+2+k-\frac{k}{q}\right)\right\rceil=\left\lceil\frac{q_0}{q_0-1}\left(q_0-1-\frac{k}{q}\right)\right\rceil\\
&=&\left\lceil q_0-\frac{k}{2q_0(q_0-1)}\right\rceil=q_0.
\end{eqnarray*}
Therefore $n_{h,j,k,1}\in \langle \mathcal{G}_1\rangle$.
\end{enumerate}
\end{proof}

Propositions \ref{delta0}, \ref{n_{h,j,k,0}}, and \ref{n_{h,j,k,1}} show that $\mathcal{G}=\mathcal{G}_1\cup \mathcal{F}_2$ is a set of generator for $H(P)=\langle\mathcal{F}_1\cup \mathcal{F}_2\rangle$.

\ \\
\textbf{Step 2: $\mathcal{G}$ is a minimal set of generators for $H(P)$.}

\begin{proposition}\label{nu_{h,k}}
For any $\nu_{h,k}\in \mathcal{G}_1$, $\nu_{h,k}\notin\langle\mathcal{G}\setminus	\{\nu_{h,k}\}\rangle$.
\end{proposition}
\begin{proof}
Suppose, by way of contradiction, that $\nu_{h,k}\in\langle\mathcal{G}\setminus	\{\nu_{h,k}\}\rangle$ for some $h,k$. Hence, $\nu_{h,k}=\sum_i n_i \nu_{h_i,k_i}+\sum_j m_j \mu_{h_j}$ for some $\nu_{h_i,k_i},\mu_{h_j}\in\mathcal{G}$ and integers $n_i,m_j\geq 0$; also, $\sum_i n_i + \sum_j m_j \geq 2$. 
Since $k_i\leq 2h_i-2$, we have
\begin{eqnarray*}
 \sum_i n_i \nu_{h_i,k_i} &=& \sum_i n_i h_i(q-1) - \sum_i n_i\left(k_i q_0 - \left\lceil\frac{k_i}{2}\right\rceil\right) + \sum_i n_i \\
&\geq& \sum_i n_i h_i(q-1) - \sum_i n_i\left((2h_i-2) q_0 - \left\lceil\frac{2h_i-2}{2}\right\rceil\right) + \sum_i n_i \\
&=& (q-2q_0)\sum_i n_i h_i + 2q_0 \sum_i n_i.\\
\end{eqnarray*}
As $\nu_{h,k}\leq q_0 q < \mu_{h_j}$ we have $m_j=0$ for any $j$.
If $\sum_i n_i h_i \leq h-1$, then $\nu_{h,k}>(h-1)q\geq \sum_i \nu_{h_i,k_i}$, a contradiction.
If $\sum_i n_i h_i \geq h+1$, then
\begin{eqnarray*}
\sum_i n_i \nu_{h_i,k_i}& \geq& (q-2q_0)(h+1)+2q_0\sum_i n_i\\
 &=& hq +2(q_0-h)q_0 + 2q_0\left(-1+\sum_i n_i\right) >hq \geq \nu_{h,k},
\end{eqnarray*}
a contradiction.
Then $\sum_i n_i h_i = h$. By direct computation, $\nu_{h,k}=\sum_i n_i \nu_{h_i,k_i}$ is equivalent to
\begin{equation}\label{equa} \left(k-\sum_i n_i k_i\right)q_0 = \left\lceil\frac{k}{2}\right\rceil+1 - \sum_i n_i\left(\left\lceil\frac{k_i}{2}\right\rceil+1\right),\end{equation}
and hence $\left\lceil\frac{k}{2}\right\rceil+1 \equiv \sum_i n_i\left(\left\lceil\frac{k_i}{2}\right\rceil+1\right) \pmod{q_0}$.
Since $1\leq \left\lceil\frac{k}{2}\right\rceil+1\leq q_0$ and $2\leq \sum_i n_i\left(\left\lceil\frac{k_i}{2}\right\rceil+1\right)\leq \sum_i n_i h_i = h \leq q_0$, this implies
\begin{equation}\label{eqbella}
\left\lceil\frac{k}{2}\right\rceil+1 = \sum_i n_i\left(\left\lceil\frac{k_i}{2}\right\rceil+1\right);
\end{equation}
hence, by \eqref{equa},
\begin{equation}\label{eqbelissima}
k=\sum_i n_i k_i.
\end{equation}
From  \eqref{eqbella} and \eqref{eqbelissima} it follows
$$ \frac{k}{2}+\frac{3}{2} \geq \left\lceil\frac{k}{2}\right\rceil+1 = \sum_i n_i\left(\left\lceil\frac{k_i}{2}\right\rceil+1\right)\geq \sum_i n_i\left(\frac{k_i}{2}+1\right)= \frac{k}{2}+\sum_i n_i, $$
a contradiction to $\sum_i n_i \geq 2$.
\end{proof}

\begin{proposition}\label{mu_{h}}
For any $\mu_{h}\in \mathcal{F}_2$, $\nu_{h,k}\notin\langle\mathcal{G}\setminus	\{\mu_h\}\rangle$.
\end{proposition}
\begin{proof}
 Suppose, by way of contradiction, that $\mu_{h}\in\langle\mathcal{G}\setminus	\{\mu_{h}\}\rangle$ for some $h\in\{q_0+1,\ldots,2q_0\}$. Hence, $\mu_{h}=\sum_i n_i \nu_{h_i,k_i}+\sum_j m_j \mu_{h_j}$ for some $\nu_{h_i,k_i},\mu_{h_j}\in\mathcal{G}$ and integers $n_i,m_j\geq 0$; also, $\sum_i n_i + \sum_j m_j \geq 2$.

\begin{itemize}
\item
Assume $\sum_i m_i\geq2$. Then
\begin{eqnarray*}
\sum_i n_i \nu_{h_i,k_i}+\sum_j m_j \mu_{h_j}& \geq& 2\left((q_0+1)q-q_0-(q_0-1)\right)\\
&=&(2q_0+2)q-3q_0-(q_0-2)>\mu_{\tilde h}
\end{eqnarray*}
for any $\tilde h$, a contradiction.

\item
Assume $\sum_i m_i =1$. Then $\mu_h=\sum_i n_i \nu_{h_i,k_i}+\mu_{\tilde h}$ for some $\tilde h<h$ and $\sum_i n_i\geq1$. Hence,
$$ (h-\tilde{h})q-2(h-\tilde{h})q_0 =\sum_i n_i\left( h_i q - k_i q_0 - \left\lfloor\frac{2h_i-2-k_i}{2}\right\rfloor \right). $$
If $\sum_i n_i h_i \leq (h-\tilde{h})-1$, then
$$ (h-\tilde{h})q-2(h-\tilde{h})q_0 =\sum_i n_i\left( h_i q - k_i q_0 - \left\lfloor\frac{2h_i-2-k_i}{2}\right\rfloor \right) <((h-\tilde{h})-1)q, $$
a contradiction.
If $\sum_i n_i h_i \geq (h-\tilde{h})+1$, then
\begin{eqnarray*}
 (h-\tilde{h})q-2(h-\tilde{h})q_0 &=&\sum_i n_i\left( h_i q - k_i q_0 - \left\lfloor\frac{2h_i-2-k_i}{2}\right\rfloor \right)\\
 &\geq& \sum_i n_i\left( h_i q - (2h_i-2) q_0  \right)\\
 &\geq& (h-\tilde h)q+q-2q_0(h-\tilde h +1)+2q_0\sum_i n_i\\
 & \geq&(h-\tilde{h})q+q-2q_0(h-\tilde h)> (h-\tilde{h})q,\\
\end{eqnarray*}
a contradiction. Therefore $\sum_i n_i h_i =h-\tilde{h}$.
By direct computation, $\mu_h=\sum_i n_i \nu_{h_i,k_i}+\mu_{\tilde h}$ is equivalent to
\begin{equation}\label{ora} \left(2(h-\tilde{h})-\sum_i n_i k_i\right)q_0=(h-\tilde{h})-\sum_i n_i\left(\left\lceil\frac{k_i}{2}\right\rceil+1\right).\end{equation}
Hence $h-\tilde{h} \equiv \sum_i n_i\left(\left\lceil\frac{k_i}{2}\right\rceil+1\right) \pmod{q_0}$; since the integers $h-\tilde{h}$ and $\sum_i n_i\left(\left\lceil\frac{k_i}{2}\right\rceil+1\right)$ are in $\left[1,\ldots,q_0-1\right]$, this implies $h-\tilde{h} = \sum_i n_i\left(\left\lceil\frac{k_i}{2}\right\rceil+1\right)$. By  \eqref{ora}, $2(h-\tilde{h})=\sum_i n_i k_i$. Thus,
$$ \sum_i n_i k_i = 2(h-\tilde{h})=2\sum_i n_i\left(\left\lceil\frac{k_i}{2}\right\rceil+1\right)\geq \sum_i n_i k_i + 2\sum_i n_i >\sum_i n_i k_i, $$
a contradiction.

\item Assume $\sum_i m_i=0$, that is, $\mu_h=\sum_i n_i \nu_{h_i,k_i}$ with $\sum_i n_i \geq2$.
If $\sum_i n_i h_i \leq h-1$, then $\mu_h>(h-1)q\geq\sum_i n_i \nu_{h_i,k_i}$, a contradiction.
If $\sum_i n_i h_i \geq h+1$, then
\begin{eqnarray*}
 \sum_i n_i \nu_{h_i,k_i}&\geq& (q-2q_0)\sum_i n_i h_i + 2q_0 \sum_i n_i \\
 &=& h(q-2q_0) + q+2q_0(-1+\sum_i n_i)\\
 &\geq& h(q-2q_0)+q+1 = \mu_h,
\end{eqnarray*}
a contradiction.
Then $\sum_i n_i h_i=h$. Let $\hat{h}=h-q_0\in\{1,\ldots,q_0\}$. By direct computation, $\mu_h=\sum_i n_i \nu_{h_i,k_i}$ is equivalent to
\begin{equation}\label{eora} \left(2\hat{h}-1-\sum_i n_i k_i\right)q_0 = \hat{h}+1 - \sum_i n_i\left(\left\lceil\frac{k_i}{2}\right\rceil+1\right);\end{equation}
hence,
\begin{equation}\label{vedo}\hat{h}+1 \equiv \sum_i n_i\left(\left\lceil\frac{k_i}{2}\right\rceil+1\right)\pmod{q_0}.\end{equation}
Since
$$2\leq \sum_i n_i\leq\sum_i n_i\left(\left\lceil\frac{k_i}{2}\right\rceil+1\right)\leq\sum_i n_i h_i =\hat{h}+q_0,$$
 \eqref{vedo} implies $\hat{h}+1 = \sum_i n_i\left(\left\lceil\frac{k_i}{2}\right\rceil+1\right)$.
By  \eqref{eora}, $2\hat{h}-1 = \sum_i n_i k_i$. Thus,
$$ \sum_i n_i k_i = 2\hat{h}-1 = 2\left(\sum_i n_i\left(\left\lceil\frac{k_i}{2}\right\rceil+1\right)-1\right)-1\geq 2\left(\sum_i n_i\left(\frac{k_i}{2}+1\right)-1\right)-1 $$
$$ = \sum_i n_i k_i + 2\sum_i n_i -3 > \sum_i n_i k_i,  $$
a contradiction.
\end{itemize}
\end{proof}

\section{Dual one-point codes from the Suzuki curve}
In this section we construct dual one-point codes from $\cS_q$. In particular,   we consider codes of type $\mathcal{C}_i=C_{\mathcal L}(mP_i,D_i)^{\bot}$, where 
\begin{equation*}
P_1\in \mathcal{S}_q(\mathbb{F}_{q^4})\setminus \mathcal{S}_q(\mathbb{F}_{q}), \qquad P_2\in \mathcal{S}_q(\mathbb{F}_{q}), \qquad D_i=\sum_{P\in \mathcal{S}_q(\mathbb{F}_{q^4})\setminus P_i }P.
\end{equation*}
For a more detailed introduction on AG codes we refer the readers to \cite{Sti}.

Also, estimates on the minimum distance are obtained using the so-called Feng-Rao function.  Let $\mathcal{X}$ be a nonsingular curve and $P$ an $\mathbb{F}_q$-rational  point of $\mathcal{X}$. Let the Weierstrass semigroup at $P$ be given by
\begin{equation}\label{OrdinamentoH}
H(P)=\{\rho_1=0<\rho_2<\rho_3<\cdots\}.
\end{equation}
For $\ell \geq0$, define the \emph{Feng-Rao function} 
$$\nu_\ell := | \{(i,j) \in \mathbb{N}_0^2 \ : \ \rho_i+\rho_j = \rho_{\ell+1}\}|.$$ 
Consider ${C}_{\ell}(P)= {C}_{\mathcal L}(P_1+P_2+\cdots+P_N,\rho_{\ell}P)^{\bot}$, with $P,P_1,\ldots,P_N$ pairwise distint points in $\mathcal{X}(\mathbb{F}_q)$. The number 
$$d_{ORD} ({C}_{\ell}(P)) := \min\{\nu_{m} \ : \ m \geq \ell\}$$
is a lower bound for the minimum distance $d({C}_{\ell}(P))$ of the code ${C}_{\ell}(P)$, called the \emph{order bound} or the \emph{Feng-Rao designed minimum distance} of ${C}_{\ell}(P)$; see \cite[Theorem 4.13]{HLP}. Also, by \cite[Theorem 5.24]{HLP}, $d_{ORD} ({C}_{\ell}(P))\geq \ell+1-g$ and equality holds if $\ell \geq 2c-g-1$, where $c=\max \{m \in \mathbb{Z} \ : \ m-1 \notin H(P)\}$ is the conductor of $H(P)$.

Note that length and dimension of $C_{\mathcal L}(D_1,\rho_{\ell}P_1)^{\bot}$ and $C_{\mathcal L}(D_2,\rho_{\ell}P_2)^{\bot}$ coincide; here, the coefficients $\rho_\ell$ of $P_1$ and $P_2$ are the $\ell$-th  non-gap at $P_1$ and the $\ell$-th  non-gap at $P_2$, respectively (see \eqref{OrdinamentoH}).

We list in Tables \ref{table_q8} and \ref{table_q32} the parameters of the codes $C_{\mathcal L}(D_i,\rho_{\ell}P_i)^{\bot}$ for which the minimum distance of $C_{\mathcal L}(D_1,\rho_{\ell}P_1)^{\bot}$ is larger than the one of $C_{\mathcal L}(D_2,\rho_{\ell}P_2)^{\bot}$, when $q\in\{8,32\}$. Computations have been made comparing the Feng-Rao designed minimum distance of both the codes and have been performed using MAGMA \cite{MAGMA}.

\begin{center}
\begin{table}
\caption{Comparison between the codes $\mathcal{C}_1$ and $\mathcal{C}_2$, $q=8$}\label{table_q8}
\begin{tabular}{|l|l|l|l|l|}
\hline
$\rho_\ell$ & $n$ & $n-\ell$ & $d(\mathcal{C}_1)$ & $d(\mathcal{C}_2)$\\
\hline
\hline
34 & 4124 & 4103 & 10 & 8\\
35 & 4124 & 4102 & 12 & 10\\
36 & 4124 & 4101 & 12 & 10\\
\hline
\end{tabular}
\end{table}
\end{center}

\begin{center}
\begin{table}
\caption{Comparison between the codes $\mathcal{C}_1$ and $\mathcal{C}_2$, $q=32$}\label{table_q32}
\begin{tabular}{|l|l|l|l|l||l|l|l|l|l|}
\hline
$\rho_\ell$ & $n$ & $n-\ell$ & $d(\mathcal{C}_1)$ & $d(\mathcal{C}_2)$&$\rho_\ell$ & $n$ & $n-\ell$ & $d(\mathcal{C}_1)$ & $d(\mathcal{C}_2)$ \\
\hline
\hline
261 & 1048824 & 1048686 & 38 & 32&
262 & 1048824 & 1048685 & 38 & 32\\
263 & 1048824 & 1048684 & 38 & 32&
264 & 1048824 & 1048683 & 38 & 32\\
265 & 1048824 & 1048682 & 40 & 32&
266 & 1048824 & 1048681 & 40 & 32\\
267 & 1048824 & 1048680 & 40 & 32&
268 & 1048824 & 1048679 & 40 & 32\\
269 & 1048824 & 1048678 & 40 & 32&
270 & 1048824 & 1048677 & 40 & 32\\
271 & 1048824 & 1048676 & 40 & 32&
272 & 1048824 & 1048675 & 40 & 32\\
273 & 1048824 & 1048674 & 40 & 32&
274 & 1048824 & 1048673 & 40 & 32\\
275 & 1048824 & 1048672 & 40 & 32&
276 & 1048824 & 1048671 & 40 & 32\\
277 & 1048824 & 1048670 & 40 & 32&
278 & 1048824 & 1048669 & 40 & 32\\
279 & 1048824 & 1048668 & 40 & 36&
280 & 1048824 & 1048667 & 40 & 36\\
281 & 1048824 & 1048666 & 40 & 36&
282 & 1048824 & 1048665 & 40 & 36\\
285 & 1048824 & 1048662 & 42 & 40&
286 & 1048824 & 1048661 & 42 & 40\\
287 & 1048824 & 1048660 & 42 & 41&
297 & 1048824 & 1048650 & 66 & 64\\
298 & 1048824 & 1048649 & 66 & 64&
299 & 1048824 & 1048648 & 66 & 64\\
300 & 1048824 & 1048647 & 66 & 64&
301 & 1048824 & 1048646 & 66 & 64\\
302 & 1048824 & 1048645 & 66 & 64&
303 & 1048824 & 1048644 & 66 & 64\\
304 & 1048824 & 1048643 & 66 & 64&
305 & 1048824 & 1048642 & 66 & 64\\
306 & 1048824 & 1048641 & 66 & 64&
307 & 1048824 & 1048640 & 66 & 64\\
308 & 1048824 & 1048639 & 66 & 64&
309 & 1048824 & 1048638 & 66 & 64\\
310 & 1048824 & 1048637 & 66 & 64&
313 & 1048824 & 1048634 & 70 & 68\\
314 & 1048824 & 1048633 & 70 & 68&
317 & 1048824 & 1048630 & 73 & 72\\
318 & 1048824 & 1048629 & 73 & 72&
321 & 1048824 & 1048626 & 80 & 76\\
322 & 1048824 & 1048625 & 82 & 76&
323 & 1048824 & 1048624 & 84 & 77\\
324 & 1048824 & 1048623 & 84 & 80&
325 & 1048824 & 1048622 & 88 & 80\\
326 & 1048824 & 1048621 & 88 & 80&
327 & 1048824 & 1048620 & 88 & 81\\
328 & 1048824 & 1048619 & 88 & 82&
341 & 1048824 & 1048606 & 97 & 96\\
342 & 1048824 & 1048605 & 97 & 96&
353 & 1048824 & 1048594 & 109 & 108\\
354 & 1048824 & 1048593 & 112 & 108&
355 & 1048824 & 1048592 & 112 & 109\\
357 & 1048824 & 1048590 & 114 & 112&
358 & 1048824 & 1048589 & 114 & 112\\
359 & 1048824 & 1048588 & 114 & 113&
361 & 1048824 & 1048586 & 118 & 116\\
362 & 1048824 & 1048585 & 118 & 116&
363 & 1048824 & 1048584 & 118 & 117\\
365 & 1048824 & 1048582 & 121 & 120&
366 & 1048824 & 1048581 & 121 & 120\\
369 & 1048824 & 1048578 & 124 & 123&
386 & 1048824 & 1048561 & 142 & 140\\
387 & 1048824 & 1048560 & 142 & 141&
389 & 1048824 & 1048558 & 145 & 144\\
390 & 1048824 & 1048557 & 145 & 144&
&&&&\\
\hline
\end{tabular}
\end{table}
\end{center}

\section*{Acknowledgments}

The research of D. Bartoli, M. Montanucci, and G. Zini was partially supported by Ministry for Education, University and Research of Italy (MIUR) (Project PRIN 2012 ``Geometrie di Galois e strutture di incidenza'' - Prot. N. 2012XZE22K$_-$005)  and by the Italian National Group for Algebraic and Geometric Structures and their Applications (GNSAGA - INdAM).

\vspace{1ex}
\noindent
Daniele Bartoli

\vspace{.5ex}
\noindent
Unviersit\`a degli Studi di Perugia,\\
Dipartimento di Matematica e Informatica,\\
Via Luigi Vanvitelli 1,\\
06123 Perugia,\\
Italy,\\
daniele.bartoli@unipg.it\\

\vspace{1ex}
\noindent
Maria Montanucci

\vspace{.5ex}
\noindent
Universit\`a degli Studi di Padova,\\
Dipartimento di Tecnica e Gestione dei Sistemi Industriali,\\
Stradella S. Nicola 3,\\
36100 Vicenza,\\
Italy,\\
montanucci@gest.unipd.it

\vspace{1ex}
\noindent
Giovanni Zini

\vspace{.5ex}
\noindent
Universit\`a degli Studi di Milano-Bicocca\\
Dipartimento di Matematica e Applicazioni,\\
Via Cozzi 55,\\
20125 Milano,\\
Italy,\\
giovanni.zini@unimib.it\\


\begin{thebibliography}{99}

\bibitem{BM} P. Beelen and M. Montanucci, A new family of maximal curves, \textit{J. London Math. Soc. (2)}, 2018, doi:10.1112/jlms.12144.

\bibitem{MAGMA} W. Bosma, J. Cannon, and C. Playoust, The Magma algebra system. {I}. The user language, \textit{J. Symbolic Comput.} {\bf 24},  235--265, 1997. 

\bibitem{DL1976} P. Deligne and G. Lusztig, Representations of reductive groups over finite fields, \textit{Ann. Math.}, {\bf 103}, 103--161, 1976.

\bibitem{FT1} R. Fuhrmann, F. Torres, On Weierstrass points and optimal curves, \textit{Rend. Circ. Mat. Palermo} (2) \textbf{51}, 25-46 (1998). 

\bibitem{GGS} A. Garcia, C. G\"uneri, and  H. Stichtenoth,  A generalization of the Giulietti-Korchm\'aros maximal curve, \textit{Advances in Geometry}, {\bf 10}(3), 427--434, 2010.

\bibitem{GS} Garcia, A., Stichtenoth, H., Elementary abelian $p$-extensions of algebraic functions fields, \textit{Manuscripta Math.} \textbf{72}, 67-79 (1991).

\bibitem{GK2009} M. Giulietti and G. Korchm\'aros, A new family of maximal curves over a finite field, \textit{Math. Ann.}, {\bf 343}, 229--245, 2009.

\bibitem{HLS1995} C. Heegard, J. Little, and K. Saints, Systematic encoding via Gr\"obner bases for a class of algebraic-geometric Goppa codes, \textit{IEEE Trans. Inf. Theory}, {\bf 41}, 1752--1761, 1995.

\bibitem{HKT} J.W.P. Hirschfeld, G. Korchm\'aros, F. Torres, {\it Algebraic Curves over a Finite Field}. Princeton Series in Applied Mathematics, Princeton (2008).

\bibitem{HLP} T. H\o holdt, J. H. van Lint, and R. Pellikaan, Algebraic geometry codes, in: \emph{Handbook of Coding Theory}, V. S. Pless, W. C. Huffman, and R. A. Brualdi, Eds. Amsterdam, The Netherlands: Elsevier, 1998, vol. 1, pp. 871--961.

\bibitem{Joyner2005} D. Joyner, An error-correcting codes package, \textit{SIGSAM Comm. Computer Algebra}, \textbf{39}(2),  65--68, 2005.

\bibitem{MEAGHER2008} S. Meagher and J. Top, Twists of genus three curves over finite fields, \textit{Finite Fields Appl.}, \textbf{16}, 347--368, 2010.

\bibitem{M} G. Matthews, Codes from the Suzuki function field, \textit{IEEE Trans. Inform. Theory} \textbf{50}(12), 3298-3302, (2004). 

\bibitem{Sti} H. Stichtenoth, \textit{Algebraic function fields and codes}, Springer, 2009.

\bibitem{SV} K.O. St\"ohr, J.F. Voloch, Weierstrass points and curves over finite fields, \textit{Proc. London Math. Soc.} \textbf{52}(3), 1--19 (1986).

\bibitem{TV1991} M.A. Tsfasman, G. Vladut, \textit{Algebraic-geometric Codes}, Kluwer, Dordrecht, (1991).



\end{thebibliography}
\end{document}